%% file: Ehrhart_Theory_Over_Abelian_Group_Rings.tex
\documentclass[12pt,reqno]{amsart}
\usepackage[english]{babel}
\usepackage[utf8]{inputenc}
\usepackage{amsmath,amsfonts,amssymb,amsthm}
\usepackage[margin=1in]{geometry}
\usepackage{hyperref}
\usepackage{cleveref}
\usepackage{graphicx}
\usepackage{xfrac}
\usepackage{color}
\usepackage{tikz,tikz-cd}
\usepackage{float}
\usepackage{ulem}
\usepackage{caption, subcaption}
\usepackage{comment}

\usepackage{cellspace} 
\usepackage{diagbox}

\usepackage{scalerel,stackengine}
\stackMath
\newcommand\reallywidehat[1]{%
\savestack{\tmpbox}{\stretchto{%
  \scaleto{%
    \scalerel*[\widthof{\ensuremath{#1}}]{\kern-.6pt\bigwedge\kern-.6pt}%
    {\rule[-\textheight/2]{1ex}{\textheight}}%WIDTH-LIMITED BIG WEDGE
  }{\textheight}% 
}{0.5ex}}%
\stackon[1pt]{#1}{\tmpbox}%
}

\makeatletter
\def\Ddots{\mathinner{\mkern1mu\raise\p@
\vbox{\kern
7\p@\hbox{.}}\mkern1mu\raise
4\p@\hbox{.}\mkern1mu\raise
7\p@\hbox{.}\mkern1mu}}
\makeatother

\newtheorem{thm}{Theorem}[section]
\Crefname{thm}{Theorem}{Theorems}
\newtheorem{lem}[thm]{Lemma}
\newtheorem{cor}[thm]{Corollary}

\newtheorem{prop}[thm]{Proposition}
\theoremstyle{definition}
\newtheorem{definition}[thm]{Definition}
\newtheorem{exam}[thm]{Example}

\theoremstyle{remark}
\newtheorem{rem}[thm]{Remark}

\newcommand{\conv}{\mathrm{conv}}

\newcommand{\commentout}[1]{}

\DeclareMathOperator{\vol}{vol}

\newcommand{\C}{\mathbb{C}}
\newcommand{\N}{\mathbb{N}}
\newcommand{\Q}{\mathbb{Q}}

\newcommand{\R}{\mathbb{R}}

\newcommand{\Z}{\mathbb{Z}}

\renewcommand{\phi}{\varphi}

\renewcommand{\epsilon}{\varepsilon}

\DeclareMathOperator{\ehr}{ehr}

\begin{document}

\title{Ehrhart Theory over Abelian Group Rings} 

\author[R. Davis]{Robert Davis}
\address{
	Colgate University \\
	Hamilton, NY, USA }
\email{rdavis@colgate.edu}

\author[J. A. De Loera]{Jes\'us A. De Loera}
\address{University of California, Davis, CA, USA}
\email{jadeloera@ucdavis.edu}

\author[A. Garber]{Alexey Garber}
\address{
University of Texas Rio Grande Valley, 
Brownsville, TX, USA}
\email{alexey.garber@utrgv.edu}

\author[K. Jochemko]{Katharina Jochemko}
\address{KTH Royal Institute of Technology, Stockholm, Sweden}
\email{jochemko@kth.se}

\author[M. Omar]{Mohamed Omar}
\address{York University, Toronto, Canada}
\email{omarmo@yorku.ca}

\author[J. Yu]{Josephine Yu}
\address{Georgia Institute of Technology, Atlanta, GA, USA}
\email{jyu@math.gatech.edu}

\keywords{
    Ehrhart theory, lattice point enumeration, generating functions, convex polytopes}
    \subjclass[2020]{52B20; 05A15; 52B45}

\date{\today}

\begin{abstract} 
We introduce a unified framework for Ehrhart theory in which lattice point enumerators take coefficients in an Abelian group ring, encoding substantially richer algebraic data than classical counts.  We prove that fundamental results of Ehrhart theory extend to this setting through a generalized Brion theorem, including rational generating functions, reciprocity phenomena, connections to volume, and vertex-cone decompositions. 
We further show how to derive $q$-enumerative and weighted theories from this setting, recasting several major refinements of Ehrhart theory as consequences of a single algebraic mechanism. We also show how our framework combines with equivariant Ehrhart theory. 
\end{abstract}
\maketitle

\section{Introduction}

Let $P$ be a $d$-dimensional lattice polytope, i.e., a $d$-dimensional convex polytope with vertices in $\Z^d$.
The {\it Ehrhart function} 
of $P$ is defined by $\ehr(P;n) = |nP\cap \Z^d|$ for nonnegative integers $n$.
The {\it Ehrhart series} of $P$ is a formal power series in the variable $t$ given by
$$\mathcal E(P;t) = \sum_{n\geq 0}\ehr(P;n)t^n=\sum_{n\geq 0} |nP\cap \Z^d|t^n.$$  
See \cite{barvinokzurichbook, beckrobins,stanley-1996-comb-comm-algebra,stanley} for details. Ehrhart functions and their series have been the subject of research for more than 60 years \cite{Ehrhart1962}, with numerous applications in combinatorics and algebra (see \cite{barvinokzurichbook, beckrobins,HibiBook} and the 
references therein). In this work, we present a new generalization of Ehrhart functions by changing the coefficients of the Ehrhart series and related generating functions. Specifically, for a given Abelian group $G$ and a commutative ring $R$ with unity, we study weighted lattice point enumeration, where each lattice point is assigned a weight in the group ring $R[G]$.

\begin{definition}
Let $G$ be an Abelian group, $\phi:\Z^d \to G$ be a homomorphism, and $R$ be a 
commutative ring with unity. We define the {\it $\phi$-Ehrhart function}  of $P$ and the {\it $\phi$-Ehrhart series} of $P$ as
$$\ehr_R(P,\varphi;n) = \sum_{\alpha \in nP \cap \Z^d}\varphi(\alpha) \in R[G] \quad \text{and} \quad \mathcal E_R(P,\varphi;t) = \sum_{n\geq 0}\ehr(P,\phi;n)t^n \in R[G][[t]],$$
respectively. In the interest of simplicity, we will routinely use $\ehr(P,\varphi;n)$ and $\mathcal E(P,\varphi;t)$ when $R$ is understood.

\end{definition}

We study the $\phi$-Ehrhart series and its properties. We will see that it unifies several earlier variations in Ehrhart theory, in particular results that take into account weights on the lattice points. We begin with an example that reveals the richness of this definition. 

\begin{exam}\label{ex:first example}
Consider the polytope $P = [0,1]^2$, the group $G = C_3=\{1,\tau,\tau^2\}$, and the ring $R = \Z$. Using the homomorphism $\phi$ defined by $\phi(e_1)=\phi(e_2)=\tau$ (see Figure~\ref{fig:modularity example}) we obtain
\[
    \ehr(P,\varphi;n) =  (1+\tau+\tau^2+\cdots+\tau^n)^2 \quad 
    \text{and} \quad \mathcal{E}(P,\varphi;t) = \frac{1+\tau t}{(1-t)(1-\tau t)(1-\tau^2t)}.
\]
In the first identity the Ehrhart function is an element of $R[G]$ for each value of $n$ (the reader should not be confused as $\tau \notin {\mathbb C}$). To see why the second identity is true, one might be tempted to abbreviate $\ehr(P,\phi;n)$ as $\left(\frac{1-\tau^{n+1}}{1-\tau}\right)^2$, however this can not be done in $R[G]$ because $1-\tau$ is a zero divisor. Instead, directly compute that
\[
\mathcal{E}(P,\phi;t) = \sum_{n \geq 0} \left(\sum_{i=0}^n \tau^i\right) \left(\sum_{j=0}^n \tau^j\right) t^n = \sum_{n \geq 0} \sum_{0 \leq i,j \leq n} \tau^{i+j} t^n \\
= \sum_{i,j \geq 0} \tau^{i+j} \sum_{n \geq \max(i,j)} t^n.
\]
Summing each inner sum as $\frac{t^{\max(i,j)}}{1-t}$ gives
\[
\mathcal{E}(P,\phi;t) = \frac{1}{1-t} \sum_{i,j \geq 0} \tau^{i+j}t^{\max(i,j)} = \frac{1}{1-t} \left( \sum_{i \geq j \geq 0} \tau^{i+j}t^i + \sum_{j>i \geq 0} \tau^{i+j} t^j\right).
\]
For the first sum we have
\[
\sum_{i \geq j \geq 0} \tau^{i+j} t^i = \sum_{j \geq 0} \tau^j \sum_{i \geq j} (\tau t)^i = \frac{1}{1-\tau t} \sum_{j \geq 0} \tau^j (\tau t)^j = \frac{1}{(1-\tau t)(1-\tau^2 t)},
\]
and a similar computation yields $\frac{\tau t}{(1-\tau t)(1-\tau^2 t)}$ for the second sum.  Adding these two yields 
\[
\mathcal{E}(P,\phi;t) = \frac{1+\tau t}{(1-t)(1-\tau t)(1-\tau^2t)}.
\]

\input{Figures/unit-square-modular-tau}

If we were to use $R = \Z_2$ instead, we may note a repeated factor $1+\tau t$ in both numerator and denominator of $\mathcal{E}(P,\varphi;t)$ as 
\[
\mathcal{E}(P,\varphi;t) = \frac{1+\tau t}{(1+t)(1+\tau t)(1+\tau^2t)}.
\]
\end{exam}

Following Example \ref{ex:first example} we could more generally 
let $G = C_r = \langle \tau \mid \tau^r = 1\rangle$ be the cyclic group of order $r$, and let $R = \Z$.
Let $\varphi: \Z^d \to G$ be a homomorphism that sends each standard basis vector to $\tau$.
The kernel of $\varphi$ has cosets $K_0,\dots,K_{r-1}$ and 
it follows that for $\theta_i(nP):=|nP\cap K_i|$,
\[
    \ehr(P,\varphi;n) = \sum_{i=0}^{r-1} \theta_i(nP)\tau^i.
\]
Table~\ref{tab:group-ring-table} summarizes the interpretations of the coefficients of $\tau^i$ in $\ehr(P,\varphi;n)$ for several natural choices of group $G$ and ring $R$. For example, one recovers counting of lattice points modulo $p$.
\input{Figures/chart-tau}

\subsection*{Our Contributions}
Our  first main result says that the $\phi$-Ehrhart series is a rational function.  We work with a general construction; for a rational polyhedron $Q \subset \R^d$, let the {\it $\varphi$-integer point transform} be
\[
\sigma_R(Q,\varphi;\mathbf{z}) = 
\sum_{\alpha \in Q \cap \Z^d} \varphi(\alpha)\mathbf{z}^\alpha
\]
where $\mathbf{z}^{\alpha} = z_1^{\alpha_1}\cdots z_d^{\alpha_d}$.
This is a formal power series in variables $z_1,\dots,z_d$ with coefficients in $R[G]$. These formal power series form a module $R[G][[\mathbf{z}]]$ over the ring of Laurent polynomials $R[G][\mathbf{z}^{\pm}]:=R[G][z_1,\dots,z_d,z_1^{-1},\dots,z_d^{-1}]$.
We again write $\sigma(Q,\varphi;\mathbf{z})$ when $R$ is understood. 

Let $\mathcal C\subset \R^{d}$ be the half-open rational simplicial cone obtained by removing $k$ of the facets
$$\mathcal C:=\left\{\sum_{i=1}^{d} \lambda_i u_i \mid  \lambda_1,\ldots,\lambda_k>0,\, \lambda_{k+1},\ldots,\lambda_d\geq 0\right\},$$
where $\{u_1,\dots,u_{d}\}$ is a linearly independent set of integer vectors.
Note that one may let any $k$ of the inequalities $\lambda_i \geq 0$ be strict and not just the first $k$, as this amounts to a reindexing of $\{u_1,\dots,u_d\}$.
The {\it fundamental parallelepiped}, $\Pi$, of this cone is
\begin{equation}
    \label{eqn:fundamental}
\Pi:=\left\{\sum_{i=1}^{d} \lambda_i u_i \mid   \lambda_1,\ldots,\lambda_k\in (0,1],\, \lambda_{k+1},\ldots,\lambda_d \in [0,1)\right\}.
\end{equation} 

\begin{thm}[Rationality]
\label{lem:brion_group}
For a half-open rational simplicial cone $\mathcal{C}$ with apex at the origin, generated by $u_1,u_2,\ldots,u_d$, we have
\[
\sigma(\mathcal{C}, \varphi; \mathbf{z}) =   \frac{\sum_{\alpha \in \Pi\cap\Z^{d}} \varphi(\alpha) \mathbf{z}^\alpha}{\prod_{i=1}^{d}(1-\varphi(u_i)\mathbf{z}^{u_i})}.
\]
\end{thm}

The equation is interpreted as the equality \[
\prod_{i=1}^{d}(1-\varphi(u_i)\mathbf{z}^{u_i})\, \cdot \,\sigma(\mathcal{C}, \varphi; \mathbf{z}) = \sum_{\alpha \in \Pi\cap \Z^d} \varphi(\alpha) \mathbf{z}^\alpha\]
in the $R[G][\mathbf{z}^{\pm}]$-module $R[G][[z]]$. 

It is known that every rational polyhedral cone can be decomposed as a (disjoint) union of finitely many half-open rational simplicial cones (see, for example, \cite{StanleyDecompositions}, where a shelling is used to say which faces are close or open). It follows that
the $\varphi$-integer point transform of any pointed rational polyhedral cone is a rational function in $R[G](\mathbf{z})$.  In other words, we can extend the discussion from~\cite[Section 11.2]{beckrobins} to the $R[G]$ setting. There is a unique map from the submodule of $R[G][[\mathbf{z}]]$ generated by integer point transforms of rational polyhedral cones to the rational functions $R[G](\mathbf{z})$ (by this we refer to the localization of $R[G][z]$ by elements of the form $1-\phi(u)\mathbf{z}^{u}$), 
which is an $R[G][\mathbf{z}^{\pm}]$-module homomorphism. Moreover, since $1-\phi(u_i)\mathbf{z}^{u_i}$ is not a zero divisor in $R[G](\mathbf z)$, we get that the non-pointed cones are mapped to 0 by this homomorphism, see \cite[Lemma 11.2]{beckrobins}. 

Theorem~\ref{lem:brion_group} can be extended to cones whose apex is an arbitrary rational (or real) point in $\R^d$ as we outline in the following theorem.

\begin{thm}[Rationality for Rational Cone]\label{thm:apex}
Let $\mathcal{C}_a$ be a half-open rational simplicial cone generated by $u_1,u_2,\ldots,u_d$ and with apex $a \in \R^d$. Then we have 
\[
\sigma(\mathcal{C}_a,\phi;\mathbf{z}) =   \frac{\sum_{\alpha \in \Pi_a \cap\Z^{d}} \varphi(\alpha) \mathbf{z}^\alpha}{\prod_{i=1}^{d}(1-\varphi(u_i)\mathbf{z}^{u_i})}.
\]
where $\Pi_a$ is the translation by $a$ of the fundamental parallelepiped $\Pi$ of the cone with apex at the origin generated by $\{u_1,\ldots,u_d\}$.
\end{thm}

As a corollary of Theorem~\ref{lem:brion_group} we can obtain the following rationality result for $\phi$-Ehrhart series. It can be obtained by applying homomorphism $\mathbf z^\alpha \mapsto t^{\alpha_{d+1}}$ in Theorem~\ref{lem:brion_group}.

\begin{cor}
\label{cor:ehrhart_group}
For a half-open simplex $\Delta$, taking $\mathcal{C}$ to be the cone generated by $\Delta \times \{1\}$ in one higher dimension, the $\varphi$-Ehrhart series of $\Delta$ is a rational function in $t$ with coefficients in $R[G]$:
\begin{equation}\label{eq:phi-ehrhart}
   \mathcal{E}(\Delta,\phi;t)=\sum_{n\geq 0} \ehr(\Delta,\varphi;n)t^n = \frac{\sum_{(\beta,m) \in \Pi\cap \Z^{d+1}} \varphi(\beta) t^m}{\prod_{i=1}^{d+1}(1-\varphi(v_i)t)}.
\end{equation}
\end{cor}

We call the numerator \[h^*(\Delta,\phi;t) := \sum_{(\beta,m) \in \Pi\cap \Z^{d+1}} \varphi(\beta) t^m\] the {\it $h^*_\phi$-polynomial} of $\Delta$. 

\begin{exam}
    Recall from Example~\ref{ex:first example} that when $P = [0,1]^2$, $R = \Z_2$, $G = C_3$, and $\varphi(e_1) = \varphi(e_2) = \tau$, we obtain
    \[
        \mathcal{E}(P,\varphi;t) = \frac{1+\tau t}{(1+t)(1+\tau t)(1+\tau^2t)} = \frac{1}{(1+t)(1+\tau^2t)}.
    \]
    We can use \eqref{eq:phi-ehrhart} to explain why the cancellation occurs and obtain the simplified rational function more directly.

    We decompose the unit square into three half-open simplices: the closed diagonal $S_1=\conv((0,0),(1,1))$, and two half-open triangles $S_2$ and $S_3$.
    
    Because of our definition of $\varphi$, there is a symmetry of $\varphi$ across the line $x=y$.
    Since $R = \Z_2$, we get
    \[
        \mathcal{E}(P,\varphi;t)= \mathcal{E}(S_1,\varphi;t) + \mathcal{E}(S_2,\varphi;t) + \mathcal{E}(S_2,\varphi;t) = \mathcal{E}(S_1,\varphi;t) + 2\mathcal{E}(S_2,\varphi;t) = \mathcal{E}(S_1,\varphi;t).
    \]
    Evaluating $\mathcal{E}(S_1,\varphi;t)$ immediately gives $1/((1+t)(1+\tau^2t))$.
\end{exam}

A classical reciprocity result in Ehrhart theory can be recovered from our framework: Ehrhart--Macdonald reciprocity states that if $P$ is a $d$-dimensional lattice polytope and $P^{\circ}$ is the topological interior of $P$, then $\mathcal{E}(P;t^{-1}) = (-1)^{d+1}\mathcal{E}(P^{\circ};t)$.
To adapt this result, let $\phi: \Z^d \to G$ be a homomorphism. By $\phi': \Z^d \to G$ we denote the homomorphism defined as $\phi'(\alpha)=(\phi(\alpha))^{-1}$ for all $\alpha\in\Z^d$. Also, for a half-open simplex $\Delta$, let $\Delta'$ denote the half-open simplex with the same vertices but with complementary facets.

The following reciprocity theorem holds for $h_{\varphi}^*$-polynomials; see also \cite[Thm. 5.8]{schilling2025}.

\begin{thm}[Reciprocity]\label{thm:phi-reciprocity}
For every homomorphism $\phi:\Z^d\to G$ and every half-open simplex $\Delta\subset \Z^d$ with vertices $v_1,\ldots,v_{d+1}$,
$$h^*(\Delta',\phi;t)=\phi(v_1+\ldots+v_{d+1})t^{d+1}h^*(\Delta,\phi';1/t).$$

More strongly,
for every simplex $\Delta$ and homomorphism $\phi$,
$$\mathcal E(\Delta,\phi';t^{-1})=(-1)^{d+1}\mathcal E(\Delta',\phi;t).$$
\end{thm}

Lastly, from the existence of Brion's identity on multivariate integral transforms of polytopes one can recover the following $\phi$-Brion version.
To each vertex $v$ of a polytope $P$, we attach a {\textit{vertex cone}}     \[K_{v,P} = \{v + \lambda(u-v) \mid u \in P, \lambda \in \R_{\geq 0}\}.\]
From \Cref{lem:brion_group}, the $\phi$-integral transform of each vertex cone is a multivariate rational function in variables $z_1,\dots, z_d$ with coefficients in $R[G]$.

\begin{thm}[$\varphi$-Brion]\label{thm:phi-Brion}
For a rational polytope $P$ we have the following identity for weighted integer point transforms as  rational functions:
\[
\sigma(P,\varphi,\mathbf{z}) =  \sum_{v \text{ vertex of }P} \sigma(K_{v,P}, \varphi, \mathbf{z}).
\]
\end{thm}

The proofs of the four theorems above are presented in Section \ref{sec:proofsofmainthms}.  

Beyond rationality and reciprocity, our framework yields new structural information on the distribution of lattice points among congruence classes of finite Abelian quotients. Suppose $\varphi:\mathbb Z^d \twoheadrightarrow G$ is a surjective homomorphism onto a finite Abelian group. Then the coefficients of the group-ring valued Ehrhart function naturally record the number of lattice points in each fiber of $\varphi$. It is interesting to look at the coefficients in front of each element of $G$. In Section~\ref{sec:classical} we prove that for every $g\in G$, the coefficient of $g$ in the $\varphi$-Ehrhart function is a quasi-polynomial in $n$ with leading term $\tfrac{\vol(P)}{|G|}$,
so that each congruence class receives asymptotically the same proportion of lattice points, and we connect to the classical leading coefficient of Ehrhart polynomials. We further show that every coefficient function is a quasipolynomial whose period is bounded by the exponent of $G$. These results provide a new interpretation of Ehrhart quasipolynomiality through finite quotient lattices and connect lattice-point enumeration with the algebraic structure of finite Abelian groups.

The paper presents several other consequences and applications of our framework. We can recover several known results by making particular choices 
of $G$ and $R$ in Sections~\ref{sec:proofsofmainthms}, ~\ref{sec:classical} and~\ref{sec:recoverweighted}. In this paper we use the group $G$ to label lattice points inside a polytope $P$, while in the equivariant Ehrhart theory \cite{StapledonEquivariant} there is a group $H$ acting on the lattice points of $P$. In the final Section \ref{sec:phiEquivariantEhrhart} we discuss how our group-ring formalism can be combined with the orbit structure under $H$, and the dilations of $P$.

\section{Proofs of Rationality, Reciprocity, and $\varphi$-Brion}
\label{sec:proofsofmainthms}

Our approach starts with the well-known multivariate version of Ehrhart series and Stanley's original approach to represent Ehrhart series for simplices through summation over cones \cite{StanleyDecompositions}. Many of our results are based on the fact that
the relations we know and love (e.g., reciprocity, Brion, etc) survive under the right maps.

Let $\Delta=\conv(v_1,\ldots,v_{d+1})$ be a $d$-dimensional half-open simplex with vertices in $\Q^d$ with first $k$ facets removed. That is, $\Delta$ is defined by
\begin{equation}
\label{eqn:Delta}
\Delta:=\left\{\sum_{i=1}^{d+1} \lambda_i v_i \mid \sum_{i=1}^{d+1} \lambda_i =1, \,\lambda_1,\ldots,\lambda_k>0, \,\lambda_{k+1},\ldots,\lambda_{d+1}\geq 0\right\}.
\end{equation}
Since every polytope can be written as a disjoint union of half-open full-dimensional simplices, we restrict our attention primarily to $\Delta$. As in the introduction, we note that we allow strict inequalities $\lambda_i>0$ for any subset of indices $i\in I, I\subset \{1, \ldots,d+1\}$ and the setting of \eqref{eqn:Delta} can be achieved by reindexing.

Let $\mathcal C\subset \R^{d+1}$ be the half-open cone over $\Delta$, i.e., the cone generated by vectors $u_i=(v_i,1)$ with the first $k$ facets removed. Also, let $\Pi$ be the fundamental parallelepiped of this cone. With these facts and notation we will now prove  \Cref{lem:brion_group}; note the minor change in dimension.

\begin{proof}[Proof of \Cref{lem:brion_group} (Rationality)] 

We will rewrite the $\phi$-integer point transform of $\mathcal C$ 
$$\sigma(\mathcal C,\phi, \mathbf z)=\sum_{\alpha\in \mathcal C\cap \Z^{d+1}}\phi(\alpha)\mathbf z^{\alpha}$$
as sum of different classes of lattice points with respect to translations by vectors $u_i$.

We can represent $\mathcal C$ as the disjoint union of translated copies of $\Pi$ as follows: 
$$\mathcal C=\bigsqcup_{c\in \N^{d+1}}\left(\Pi+\sum_{i=1}^{d+1} c_iu_i\right),$$ where $c = (c_1,\dots,c_{d+1})$. For every $\alpha\in \mathcal{C}\cap \Z^{d+1}$, let $\beta$ be the corresponding representative in~$\Pi$. Then 
$\alpha=\beta+\sum_{i=1}^{d+1} c_i u_i$
for a certain choice of nonnegative integers $c_i$, and all options for $c_i$ will be present among the points in $\mathcal C\cap \Z^{d+1}$. Using this in the formula for the $\phi$-integer point transform, we get
\begin{equation}\label{eq:sum}\begin{aligned}
    \sum_{\alpha\in \mathcal C\cap \Z^{d+1}}\phi(\alpha)\mathbf z^\alpha &= \sum_{\beta\in \Pi\cap \Z^{d+1}}\left(\sum_{c\in \N^{d+1}}\phi\left(\beta + \sum_{i=1}^{d+1} c_iu_i\right)\mathbf z^{\beta+\sum c_iu_i}\right) \\
    &= \left(\sum_{\beta\in \Pi\cap \Z^{d+1}}\phi(\beta)\mathbf z^\beta\right) \sum_{c\in\N^{d+1}}\left(\phi\left(\sum_{i=1}^{d+1} c_i u_i\right)\mathbf z^{\sum c_iu_i}\right).
\end{aligned}
\end{equation}

Using that $\phi$ is a homomorphism we transform
$$\begin{aligned}
    &\sum_{c \in \N^{d+1}}\left(
    \phi \left(\sum_{i=1}^{d+1} c_i u_i\right)\mathbf z^{\sum c_iu_i}\right)
    = \sum_{c \in \N^{d+1}}\left(\prod_{i=1}^{d+1} (\phi(u_i))^{c_i}(\mathbf z^{u_i})^{c_i} \right)\\
    &=\prod_{i=1}^{d+1}(1+\phi(u_i)\mathbf z^{u_i}+(\phi(u_i))^2(\mathbf z^{u_i})^2+\ldots)=\frac{1}{\prod_{i=1}^{d+1}(1-\phi(u_i)\mathbf z^{u_i})}.
\end{aligned}$$

Combining with \eqref{eq:sum} we get the claimed equality.
\end{proof}

\begin{rem}
We note that there is an alternative way to get the same result directly from the classical integer point transform formula. Rewrite the classical formula for integer point transforms
as an equality of formal (Laurent) series
$$\left(\sum_{\alpha\in \mathcal C\cap \Z^{d+1}}\mathbf z^{\alpha}\right)
\left(\Pi(1-\mathbf z^{u_i})\right)=\sum_{\beta\in\Pi}\mathbf z^\beta$$ with integer 
coefficients. We apply two homomorphisms, the ring homomorphism from $\Z$ to $R$ that takes $1$ to the unity of $R$, and the group homomorphism $\mathbf z^\alpha \mapsto \phi(\alpha)\mathbf z^\alpha$ from the group $Z_{d+1}$ consisting of all powers of $\mathbf z$ (isomorphic to $\Z^{d+1}$) to the direct product $G\times Z_{d+1}$. The result then follows from the equality of formal (Laurent) series with coefficients in $R[G]$.
\end{rem}

For a proof of Theorem \ref{thm:apex}, we need to make a minor adjustment.

\begin{proof}[Proof of \Cref{thm:apex} (Rationality for Rational Cone)]
We note that copies of the parallelepiped $\Pi_a$ translated by natural linear combinations of the vectors $u_1,\ldots,u_d$ give a tiling of the cone $\mathcal C_a$. After that the computations from the proof of \Cref{lem:brion_group} can be repeated verbatim replacing $\Pi$ with $\Pi_a$.
\end{proof}

We continue with the reciprocity Theorem \ref{thm:phi-reciprocity}. Let $\phi: \Z^d \to G$ be a homomorphism. Let $\phi': \Z^d \to G$ be the homomorphism defined as $\phi'(\alpha)=(\phi(\alpha))^{-1}$ for $\alpha\in\Z^d$.  For a half-open simplex $\Delta$, let $\Delta'$ be the half-open simplex on the same vertices with complementary facets removed. In particular, if $\Delta$ is a closed simplex, then $\Delta'$ is its interior.
We establish the following reciprocity property for $h^*$-polynomials.

$$h^*(\Delta',\phi;t)=\phi(v_1+\ldots+v_{d+1})t^{d+1}h^*(\Delta,\phi';1/t).$$

\begin{proof}[Proof of \Cref{thm:phi-reciprocity} (Reciprocity)]

Let $v_1,\ldots,v_{d+1}$ be vertices of $\Delta$ and $\Delta'$. Corollary~\ref{cor:ehrhart_group} gives
$$\mathcal E(\Delta,\phi;t)=\frac{\sum_{(\beta,m)\in \Pi\cap\Z^{d+1}}\phi(\beta)t^m}{\prod_{i=1}^{d+1} (1-\phi(v_i)t)} \quad \text{and} \quad \mathcal E(\Delta,\phi';1/t)=\frac{\sum_{(\beta,m)\in \Pi\cap\Z^{d+1}}(\phi(\beta))^{-1}t^{-m}}{\prod_{i=1}^{d+1} (1-(\phi(v_i))^{-1}t^{-1})}.$$

Multiplying both numerator and denominator of the $\phi'$-Ehrhart series $\mathcal E(\Delta,\phi';1/t)$ by $t^{d+1}\phi(v_1+\ldots+v_{d+1})$ and bringing to the previous denominator yields
$$
\mathcal E(\Delta,\phi';1/t)=(-1)^{d+1}\frac{\sum_{(\beta,m)\in \Pi\cap \Z^{d+1}}\phi(v_1+\ldots+v_{d+1}-\beta)t^{d+1-m}}{\prod_{i=1}^{d+1} (1-\phi(v_i)t)}.$$

The fundamental parallelepipeds $\Pi$ and $\Pi'$ for $\Delta$ and $\Delta'$ are reflections of each other with respect to $\frac12(v_1+\ldots+v_{d+1},d+1)$. 
Thus, the last numerator is the $h^*_\phi$-polynomial of $\Delta'$ and the statement of the theorem follows.
\end{proof}

\begin{proof}[Proof of \Cref{thm:phi-Brion} ($\varphi$-Brion)]

The famous Brion's theorem or Brion's identity says (see e.g., \cite{barvinokzurichbook,beckrobins}) $$\sigma(P,1,\mathbf{z}) =  \sum \sigma(K_{v,P},1,\mathbf{z})$$ over $\mathbb{Z}(\mathbf{z})$, where the sum is over the vertices of $P$.  This is an identity of rational functions in the variables in $\mathbf{z}$. Now apply to the identity the ring homomorphism $\Z(\mathbf{z}) \to R[G](\mathbf{z})$ given by $z_i \to \phi(e_i)z_i$ for each $i$. For each $\alpha \in \Z^d$, this sends $\mathbf{z}^{\alpha}$ to $\phi(\alpha)\mathbf{z}^{\alpha}$. Since $\phi$ is a 
group homomorphism, this extends to Laurent monomials and Laurent polynomials. 

It is important to note that this is well-defined on the rational functions appearing in Brion's theorem. The rational functions coming from vertex cones 
are obtained, after decomposing pointed rational cones into half-open rational simplicial cones, from denominators of the form $1-\mathbf{z}^u$. Under the homomorphism $\mathbf{z}^{\alpha} \mapsto \phi(\alpha)\mathbf{z}^{\alpha}$, such a factor becomes $1-\phi(u)\mathbf{z}^u$. These are invertible in $R[G](\mathbf{z})$. 

Now apply the homomorphism on both sides of the classical Brion identity. On the left hand side, 
\[
\sigma(P,1,\mathbf{z}) = \sum_{\alpha \in P \cap \Z^d} \mathbf{z}^{\alpha} \mapsto  \sum_{\alpha \in P \cap \Z^d} \phi(\alpha) \mathbf{z}^{\alpha} = \sigma(P,\phi,\mathbf{z}).
\]
On the right-hand side, each ordinary vertex-cone transform $\sigma(K_{v,P},1,\mathbf{z})$ is sent to the corresponding weighted vertex-cone transform $\sigma(K_{v,P},\phi,\mathbf{z})$. The result then follows.
\end{proof}

\begin{exam}\label{exam:brion}
We verify Theorem~\ref{thm:phi-Brion} with Example~\ref{ex:first example}. Recall that $P=[0,1]^2$, $R=\Z$, $G=\langle \tau \ | \ \tau^3=1\rangle$, and $\phi:\Z^2 \to G$ is given by $\phi(i,j)=\tau^{i+j}$.
Figure~\ref{fig:brion example} illustrates the four cones $K_{v,P}$ for which we will compute $\sigma$.

\input{Figures/brion-ex-tau}

The lattice points of $P$ are $(0,0),(1,0),(0,1),(1,1)$ so we directly have
\[
\sigma(P,\phi,\mathbf{z})=1+\tau z_1 + \tau z_2 + \tau^2 z_1z_2.
\]
Now compute $\sigma(K_{v,P},1,\mathbf{z})$ for the four vertices of $P$. The cone generated at $(0,0)$ is all points in the nonnegative quadrant so by \Cref{thm:apex} we have 
\[
\sigma(K_{(0,0),P},\phi,\mathbf{z})=\frac{1}{(1-\tau z_1)(1-\tau z_2)}.
\]
The cone at $(1,0)$ is generated by $-e_1$ and $e_2$ so again by \Cref{thm:apex}, 
\[
\sigma(K_{(1,0),P},\phi,\mathbf{z}) = \frac{\tau z_1}{(1-\tau^2 z_1^{-1})(1-\tau z_2)},
\]
as this translates the cone generated by $-e_1$ and $e_2$ at the origin by $(1,0)$. 
Similar computations reveal that
\[
\sigma(K_{(0,1),P},\phi,\mathbf{z}) = \frac{\tau z_2}{(1-\tau z_1)(1-\tau^2 z_2^{-1})}, \qquad \sigma(K_{(1,1),P},\phi,\mathbf{z}) = \frac{\tau^2 z_1z_2}{(1-\tau^2 z_1^{-1})(1-\tau^2z_2^{-1})}.
\]
Therefore, the right-hand side of Theorem~\ref{thm:phi-Brion} is
\[
\frac{1}{(1-\tau z_1)(1-\tau z_2)} + \frac{\tau z_1}{(1-\tau^2 z_1^{-1})(1-\tau z_2)} + \frac{\tau z_2}{(1-\tau z_1)(1-\tau^2 z_2^{-1})} + \frac{\tau^2 z_1z_2}{(1-\tau^2 z_1^{-1})(1-\tau^2z_2^{-1})}.
\]

To reconcile this with the left-hand side of Theorem~\ref{thm:phi-Brion}, set $x=\tau z_1$ and $y=\tau z_2$. Then $x^{-1}=\tau^2z_1^{-1}$ and $y^{-1}=\tau^2z_2^{-1}$. The right-hand side of Theorem~\ref{thm:phi-Brion} is then
\[
\frac{1}{(1-x)(1-y)} + \frac{x}{(1-x^{-1})(1-y)} + \frac{y}{(1-x)(1-y^{-1})} + \frac{xy}{(1-x^{-1})(1-y^{-1})}.
\]
This is exactly the classical Brion identity for the ordinary integer point transform for $P$ in $x,y$ variables, so it simplifies to 
\[
1+x+y+xy=1+\tau z_1 + \tau z_2 + \tau^2 z_1 z_2 = \sigma(P,\phi,\mathbf{z}),
\]
as desired. Additionally, our setting for the variables $x$ and $y$ is related to the homomorphism we used in the proof of Theorem~\ref{thm:phi-Brion}.
\end{exam}

We note that a similar homomorphism-based approach was recently used by Beck and Kunze~\cite{beck2025} to generalize results of Chapoton \cite{chapoton_2016}; see also Section~\ref{sec:recoverweighted} below.

\section{Group ring coefficients, quasipolynomiality, and volumes} 
\label{sec:classical}

 In this section, we study the coefficients of separate group elements in the $\ehr(P,\phi,n)$ and generalize the classical Ehrhart quasipolynomiality and connection to the volume of $P$. As a consequence, we recover the classical results on Ehrhart quasipolynomials for integer and rational polytopes from our $\varphi$-Ehrhart theory. 

In the remainder of the section we will assume that the group $G$ is finite and that the homomorphism $\varphi: \Z^n \rightarrow G$ is surjective. More concretely, we assume $R=\Z$. 
This is for convenience because one could work over any ring $R$ of characteristic zero and arrive at similar results (this hinges on the fact that $\Z$ has an injection into $R$ in this case). 

Recall that the classical Ehrhart function is a (quasi)polynomial with rational coefficients. We would like to give a similar interpretation to the $\varphi$-Ehrhart function $\ehr(P,\varphi;n)$ as a function of $n$. We fix a group element $g\in G$ and consider the coefficient of $g$ in $\ehr(P,\varphi;n)$ as a function of $n$, which we denote by $[g]\ehr(P,\varphi;n)$.  We will show that it is quasipolynomial in $n$ and that the leading coefficient is equal to $\frac{\vol(P)}{|G|}$ for all $g\in G$, where $\vol(P)$ be the normalized volume of $P$ so that the fundamental domain for $\Z^d$ has volume one. We will also derive a statement on a classical (unweighted) lattice point enumeration problem.

Recall that the exponent $\exp(G)$ of a group $G$ is the smallest positive integer $m$ such that $g^m=1$ for all $g \in G$.  For finite groups, the exponent divides the order of the group.

\begin{thm}\label{thm:quasipolynomial}
Let $P \subseteq \R^d$ be a $d$-dimensional lattice polytope, $G$ be finite, and $\varphi:\Z^d \twoheadrightarrow G$ be surjective. Then for any $g \in G$, the function $[g]\ehr(P,\varphi;n)$ is quasipolynomial in $n$ of degree $d$ and period that divides $\exp(G)$.
\end{thm}
\begin{proof} 
Let $G_P:=\langle \phi(x) \colon x \in P \cap \Z^d \rangle$ be the subgroup of $G$ generated by images of lattices points in $P$. Let $q_P = \exp(G_P)$. Since $G_P$ is a subgroup of $G$, $q_P$ divides $\exp(G)$. 

Now triangulate $P$ into lattice simplices whose vertices lie in $P \cap \Z^d$. To avoid overcounting lattice points on shared faces, we turn the triangulation into a partition into half-open simplices $\Delta$ each of which has a certain number of facets removed (see, e.g., \cite{StanleyDecompositions}). 
Hence, for each $n$, the lattice points of the $n$th dilates of the simplices partitions $nP \cap \Z^d$ and hence
\[
\mathcal{E}(P,\varphi;t)=\sum_{\Delta} \mathcal{E}(\Delta,\varphi;t),
\]
where the sum is over all half-open simplices in the triangulation. Now consider a full-dimensional half-open simplex $\Delta=\conv(v_0,\ldots,v_d)$ with vertices in $P \cap \Z^d$. By Corollary~\ref{cor:ehrhart_group},
\[
\mathcal{E}(\Delta,\varphi;t) = \frac{h^*_{\varphi}(\Delta;t)}{\prod_{i=0}^d (1-\phi(v_i)t)},
\]
where $h^*_{\varphi}(\Delta;t) \in R[G][t]$ is obtained from the lattice points in the fundamental parallelepiped of the cone over $\Delta$. Each vertex $v_i$ lies in $P \cap \Z^d$ so $\phi(v_i) \in G_P$ for each $i$. Moreover, we have
\[
(1-\phi(v_i)t)(1+\phi(v_i)t+(\phi(v_i)t)^2+\cdots+(\phi(v_i)t)^{q_P-1}) = 1-t^{q_P},
\]
so multiplying the numerator and denominator of $\mathcal{E}(\Delta,\varphi;t)$ by the appropriate factor we have
\[
\mathcal{E}(\Delta,\varphi;t) = \frac{A_{\Delta}(t)}{(1-t^{q_P})^{d+1}}
\]
for some polynomial $A_{\Delta}(t) \in R[G][t]$. Summing over all half-open simplices,
\[
\mathcal{E}(P,\phi;t)=\frac{A_P(t)}{(1-t^{q_P})^{d+1}}
\]
for some polynomial $A_P(t) \in R[G][t]$. For each $g \in G$, let $A_g(t)=[g]A_P(t) \in \Z[t]$. Then for a fixed $g \in G$ we have
\[
[g]\ehr(P,\varphi;n)=[t^n] \frac{A_g(t)}{(1-t^{q_P})^{d+1}}.
\]
It follows that $[g]\ehr(P,\varphi;n)$ is given by a distinguished polynomial of degree $d$ with rational coefficients on each residue class modulo $q_P$.
\end{proof}

From Theorem~\ref{thm:quasipolynomial} and its proof we obtain the following generating function identity for counting lattice points in shifted rational polytopes. Observe that for $v=0$ and $q=1$ we regain the usual Ehrhart polynomial of a lattice polytope.
For any rational polytope $P$, the denominator of $P$ is the smallest positive integer $q$ such that $qP$ is a lattice polytope.

\begin{thm}
    \label{prop:translation}
Let $P$ be a rational polytope with denominator $q$ and let $v$ be a rational vector with denominator $q$. Then
$$\sum_{n\geq 0}|(nP-v)\cap \Z^d|t^n=\frac{h_v^*(t)}{(1-t^q)^{d+1}},$$
where $h_v^*(t)$ is a polynomial with nonnegative integer coefficients. In particular, if $P$ is a lattice polytope and $v$ an integer vector, we recover the classical Ehrhart series
$$\mathcal E(P;t) = \sum_{n\geq 0}\ehr(P;n)t^n=\frac{h^*(t)}{(1-t)^{d+1}} \, ,$$
where $h^*(t)$ is the $h^*$-polynomial of the polytope $P$.
\end{thm}

\begin{proof}
    We consider the group $G=(\Z_q)^d$ with generators $g_i$ for copies of $\Z_q$. Define the homomorphism $\phi$ as $\phi(\alpha)=\phi(\alpha_1,\ldots,\alpha_d):=\prod g_i^{\alpha_i}$, and $R=\Z$. For a rational polytope $P$ with denominator $q$ we notice that counting integer points in $nP-v$ is the same as counting points of the coset $q\Z +qv$ in $nQ$ where $Q:=qP$. Let $g=\varphi (qv)$. Then, with the notation in the proof of Theorem~\ref{thm:quasipolynomial}, $h^\ast _v=[g]A_P$. From the proof it can also be seen that for any half-open simplex $\Delta$, the numerator of $\mathcal{E}(\Delta,\varphi;t)$ is of the form $\sum_{g\in G} a _g g$ with $a_g\geq 0$ for all $g\in G$. The subsequent multiplication of the denominator and numerator, and summing over all half-open simplices in the triangulation, does not introduce negative coefficients. In particular, $[g]A_P=h^\ast _v$ has only nonnegative coefficients, and the claim follows. 
\end{proof}

In a similar but different setting, the authors of \cite{HMY} showed how Ehrhart quasi-polynomials change as the polytope~$P$ is translated by a rational vector~$v$.  In other words, they study the function
$|n(P-v)\cap \Z^d|$.

\begin{thm}\label{thm:volume}
Let $P \subseteq \R^d$ be a $d$-dimensional lattice polytope, $G$ be finite, and $\varphi:\Z^d \twoheadrightarrow G$ be surjective. As a polynomial function  in $n$, the leading coefficient of $\ehr(P,\varphi;n)$ is  \[\vol(P) \cdot \frac{1}{|G|} \sum_{g \in G} g.\] 
Equivalently, for any $g \in G$,
\[
\lim_{n \to \infty} \frac{[g]\ehr(P,\varphi;n)}{n^d}= \frac{\vol(P)}{|G|}.
\]
\end{thm}
This is curious, as the element $e_G=\frac{1}{|G|} \sum_{g \in G} g$ is a central idempotent in $R[G]$ and classically plays a significant role in understanding the representation theory of $G$.

\begin{exam}\label{ex:quasi}
Let $P=[0,1]^2$, $G=C_3=\langle \tau \rangle$ with $\tau$ a primitive third root of unity, and $R=\mathbb{Z}$. Furthermore, let $\phi:\mathbb{Z}^2 \to G$ be given by $\phi(a,b)=\tau^{a+b}$. It is immediate that $[\tau^r]\ehr(P,\varphi;n)$ counts the number of pairs $(a,b) \in \{0,1,\ldots,n\}^2$ such that $a+b \equiv r \ (\text{mod } 3)$. From this we see that the values $[\tau^r]\ehr(P,\varphi;n)$ are as in the table below,
\[
\begin{array}{c|ccc}
n & [1]\ehr(P,\phi;n) & [\tau]\ehr(P,\phi;n) & [\tau^2]\ehr(P,\phi;n) \\ \hline
& & \\
n \equiv 0 \text{ (mod } 3) & \dfrac{n^2+2n+3}{3} & \dfrac{n(n+2)}{3} & \dfrac{n(n+2)}{3} \\
& & \\
n \equiv 1 \text{ (mod } 3) & \dfrac{n(n+2)}{3} & \dfrac{n^2+2n+3}{3} & \dfrac{n(n+2)}{3} \\
& & \\
n \equiv 2 \text{ (mod } 3) & \dfrac{(n+1)^2}{3} & \dfrac{(n+1)^2}{3} & \dfrac{(n+1)^2}{3}
\end{array}
\]
We notice then that for each $r$,
\[
\frac{[\tau^r]\ehr(P,\varphi;n)}{n^2} = \frac{\dfrac{n^2}{3}+O(n)}{n^2} = \frac{1}{3} + O\left(\dfrac{1}{n}\right) \xrightarrow{n \to \infty} \frac{1}{3}, 
\]
which is compatible with \Cref{thm:volume} because $\vol(P)=1$ and $|G|=3$.

We remark to the reader a subtlety: for all elements of $G$, the coefficients of the $\phi$-Ehrhart function are integers such that there exists a (quasi)polynomial with rational coefficients that takes the same values.
\end{exam}

\begin{proof}[Proof of Theorem~\ref{thm:volume}]
For simplicity of notation, define $\psi_g(n):=[g]\ehr(P,\phi;n)$ for any $g \in G$.  Fix $a_g \in \Z^d$ as any lattice point in the preimage of $g$ under $\phi$. Then we have
\[
\psi_g(n) =|\{\alpha \in nP \cap \Z^d \colon \phi(\alpha)=g\}|= |\{\alpha \in nP \cap \mathbb{Z}^d \colon \alpha \in a_g + \ker(\varphi)\}| = |nP \cap (a_g+\ker(\varphi))|.
\]
The second equality follows from the fact that $\phi(\alpha)=g$ if and only if $\alpha \in a_g+\ker(\varphi)$.

Since $\phi$ is surjective, $\ker(\phi)$ is a full-rank sublattice of $\Z^d$ of index $|G|$. Let $B$ be a $d \times d$ integer matrix whose columns form a $\Z$-basis of $\ker(\phi)$. Then $|\det(B)|=[\Z^d:\ker(\phi)]=|G|$. 
However $\alpha \in nP \cap (a_g+\ker(\phi))$ if and only if $\alpha \in nP$ and $\alpha=a_g+Bv$ for some $v \in \Z^d$. 
The invertibility of $B$ over $\Q$  gives 
\[
\psi_g(n) = |\{v \in \Z^d \colon v \in B^{-1}(nP-a_g)\}|.
\]
Let $Q$ be the full dimensional rational polytope obtained by applying the transformation $B^{-1}$ to $P$. Then we have $B^{-1}(nP-a_g)=nQ-B^{-1}a_g$ which is a translation of $nQ$ by the vector $B^{-1}a_g \in \mathbb{Q}^d$. We now have
\[
\psi_g(n)=|(nQ-B^{-1}a_g) \cap \mathbb{Z}^d|.
\]
Then it follows from Theorem~\ref{prop:translation} that $\psi_g(n)$ is a quasipolynomial in $n$ of degree $d$. Its leading coefficient is equal to 
\[
\lim_{n\rightarrow \infty} \frac{\psi_g(n)}{n^d} = \lim_{n\rightarrow \infty} \frac{1}{n^d} \left| \left(Q-\frac{1}{n}B^{-1}a_g\right) \cap \left(\frac{1}{n}\Z\right)^d\right| = \vol(Q).
\]

Finally, $\vol(Q)=\vol(B^{-1}P)=|\det(B^{-1})|\vol(P)=\dfrac{\vol(P)}{|G|}$. The result follows.
\end{proof}

We illustrate elements of this theorem with an example. 

\begin{exam}\label{exam:quasipolynomial}
Let \[
P=\text{conv}\{(0,0),(1,0),(2,2),(0,2)\}\subseteq \mathbb R^2,\] and let $G=\mathbb Z/2\mathbb Z\oplus \mathbb Z/3\mathbb Z$. We write elements of $G$ additively, and define $\varphi:\mathbb Z^2\to G$ by $\varphi(x,y)=(x \bmod 2,\ y \bmod 3)$. This homomorphism is surjective, and the exponent of $G$ is $6$. Triangulate $P$ by drawing the diagonal from $(0,0)$ to $(2,2)$. This gives the two lattice triangles $\Delta_1=\text{conv}\{(0,0),(1,0),(2,2)\}$ and $\Delta_2=\text{conv}\{(0,0),(2,2),(0,2)\}$. These two triangles meet along the diagonal, so one should not simply add their lattice-point counts without making a convention. As in the proof of the theorem, choose a half-open triangulation by assigning the shared diagonal to $\Delta_1$ and deleting it from $\Delta_2$. Then, for every dilation factor $n$, the lattice points of $nP$ are partitioned by the corresponding half-open dilates of these two triangles. This is the elementary geometric reason that the group-valued Ehrhart series of $P$ is the sum of the group-valued Ehrhart series of the two half-open triangles.

Let us compute one coefficient directly. The dilate $nP$ is described by the inequalities $0\leq y\leq 2n$ and $0\leq x\leq n+y/2$. The coefficient of the identity element $(0,0)\in G$ counts those lattice points in $nP$ for which $x$ is even and $y$ is divisible by $3$. Such a point has the form $(x,y)=(2i,3j)$ with $i,j\geq 0$. Substituting this into the inequalities for $nP$ gives $0\leq j\leq \lfloor 2n/3\rfloor$ and $0\leq i\leq \lfloor (2n+3j)/4\rfloor$. Thus the identity coefficient is an ordinary lattice-point count with congruence conditions imposed by $\varphi$, giving rise to the polynomials below.

\[
\begin{array}{c|c}
n & [(0,0)]\ehr(P,\varphi;n) \\ \hline
n \equiv 0 \text{ (mod } 6) & (3n^2+7n+6)/6 \\
n \equiv 1 \text{ (mod } 6) & (n^2+n)/2 \\
n \equiv 2 \text{ (mod } 6) & (3n^2+5n+2)/6 \\
n \equiv 3 \text{ (mod } 6) & (3n^2+7n+6)/6 \\
n \equiv 4 \text{ (mod } 6) & (n^2+n)/2 \\
n \equiv 5 \text{ (mod } 6) & (3n^2+5n+2)/6.
\end{array}
\]
Therefore $[(0,0)]\ehr(P,\varphi;n)$ is a quasipolynomial of degree $2$ and period dividing $6$. Its leading term, written as a function of $n$, is $\frac12 n^2$. This agrees with the expected leading coefficient. Indeed, the area of $P$ is $3$, and the six congruence classes share the leading area contribution equally, giving $\text{area}(P)/|G|=3/6=1/2$. If we selected a different group element $g=(a,b)$, where $a\in\{0,1\}$ and $b\in\{0,1,2\}$, then a point mapping to $g$ has the form $(a+2i,b+3j)$ with $i,j\geq 0$. The inequalities defining $nP$ again become a pair of linear inequalities in $i$ and $j$, with floors whose behavior is controlled by residues modulo $6$. Thus every coefficient $[g]\ehr(P,\varphi;n)$ is a quasipolynomial with period dividing $\exp(G)=6$, as asserted by Theorem~\ref{thm:quasipolynomial}.
\end{exam}

\section{Weights in Lattice Points via $\phi$-Ehrhart functions} \label{sec:recoverweighted}

We turn to weighted versions of Ehrhart functions when the lattice points have an associated weight value, 
see \cite{Weighted_Bajo_etal,chapoton_2016,quasiweightsfpsac2023,deloera2024weighted} and the references therein. Here, we think of $\phi$ as a weight and recover several known results in weighted Ehrhart theory.

\subsection{Chapoton's $q$-Ehrhart}

First we recover and generalize Chapoton's $q$-Ehrhart setting~\cite{chapoton_2016}; see also \cite{bajo2026} for a recent application to $q$-chromatic polynomials of graphs.
Let $q$ be an indeterminate variable and let $(G_q, \cdot) \cong (\R,+) $ be the multiplicative group of real powers of $q$. For every linear function $\ell:\R^d\to \R$, the map $\phi_q:\Z^d\to G_q$ defined as $\phi_q(x)=q^{\ell(x)}$ is a homomorphism. Combined with Corollary~\ref{cor:ehrhart_group}, this gives the following result.

\begin{prop}[Chapoton, \cite{chapoton_2016}]
If $\ell$ is a linear form with integer coefficients which is positive on all vertices of simplex $\Delta=\conv(v_1,\ldots,v_{d+1})$, then
$\mathcal E(\Delta, \phi_q;t)$ is a rational function in $t$ and $q$ with denominator $\prod (1-q^{\ell(v_i)}t)$.
\end{prop}
\begin{proof}
The proposition is immediate from Corollary~\ref{cor:ehrhart_group}. The numerator is a polynomial in $t$ and $q$ because $\ell(v_i)$ is a nonnegative integer.
\end{proof}

The proposition can be generalized to the case of arbitrary polytopes by using half-open triangulations. For the case of a simplex and certain special families of polytopes \cite{deloera2024weighted} the nonnegativity property for the numerator continues to hold, but it fails for general polytopes as noted in~\cite{chapoton_2016}.

Our approach does not require $\ell$ to have integer coefficients, but we forfeit the polynomiality of the numerator in exchange. In the following, let $G_q$ denote the multiplicative group of $R$-powers of a variable $q$; this group is isomorphic to the Abelian group $(R,+)$.

\subsection{Polynomial Weights}

Next, we generalize Corollary~\ref{cor:ehrhart_group} to the case when each point in $\Z^d$ is weighted by its evaluation of a given polynomial in $R[x_1,\ldots,x_d]$ in addition to the weight from $G$.

\begin{definition}
Let $w$ be a polynomial in $R[x_1,\ldots,x_d]$. We define the {\it $(\phi,w)$-Ehrhart function of $P$} in $R[G]$ (left) and the {\it $(\phi,w)$-Ehrhart series of $P$} in $R[G][[t]]$ (right) as
$$\ehr(P,\phi,w;n):=\sum_{\alpha\in nP\cap \Z^d}w(\alpha)\phi(\alpha) \quad \text{and} \quad \mathcal E(P,\phi,w;t):=\sum_{n\geq 0} \ehr(P,\phi,w;n)t^n.$$
\end{definition}

Our goal is to show that $(\phi,w)$-Ehrhart series can also be written as a rational function in $R[G](t)$. We start with an example showing how the process works and then proceed with the general case.

\begin{exam}\label{ex:triangle-polynomial}
    Let $\Delta$ be the triangle in $\R^2$ with vertices $(0,0)$, $e_1 = (1,0)$, and $e_1 + e_2 = (1,1)$ with the edge lying on $x_1 = x_2$ removed.
    Here, the fundamental parallelepiped $\Pi\subset \R^3$ contains the single integer point $(1,0,1)$.
    Let $w(x_1,x_2) = x_1$ and define $\psi: \Z^2 \to G_q \times G$ to be $\psi(x_1,x_2) = q^{x_1}\varphi(x_1,x_2)$.
    From Corollary~\ref{cor:ehrhart_group}, we may directly write
    \[
        \mathcal{E}(\Delta, \psi;t) = \sum_{n\geq 0}\left(\sum_{\alpha\in n\Delta\cap \Z^d}q^{w(\alpha)}\phi(\alpha)\right)t^n = \frac{q\varphi(e_1) t}{(1-t)(1-q\phi(e_1)t)(1-q\varphi(e_1+e_2)t)}.
    \]
    To obtain $\mathcal{E}(\Delta, \varphi, w;t)$ from this, note that every coefficient of the power series in the left-hand side and in the numerator or denominator polynomials in the right-hand side can be treated as a sum of $R$-powers of $q$ each with a coefficient from $R[G]$. 
    
    The equality stays true if we formally differentiate with respect to $q$ using the rule $\frac{d}{dq}\left(q^x\right)=xq^{x-1}$ for $x\in R$, and using the standard quotient rule $(f/g)'=(f'g-fg')/g^2$ in the right-hand side.

    Doing this, we obtain 
    $$\sum_{n\geq 0}\left(\sum_{\alpha\in nP\cap \Z^d}w(\alpha)q^{w(\alpha)-1}\phi(\alpha)\right)t^n$$
    from the power series; evaluating at $q = 1$ returns $\mathcal{E}(\Delta,\varphi,w;t)$.
    On the side of rational functions, we get
    \[
        \frac{(1-t)t\varphi(e_1)(1-\varphi(e_1)\varphi(e_1+e_2)q^2t^2)}{(1-t)^2(1-q\phi(e_1)t)^2(1-q\varphi(e_1+e_2)t)^2}.
    \]
    Evaluating at $q=1$ this time gives us the answer we seek:
    \[
        \mathcal{E}(\Delta,\varphi,w;t) = \frac{(1-t)t\varphi(e_1)(1-\varphi(e_1)\varphi(e_1+e_2)t^2)}{(1-t)^2(1-\phi(e_1)t)^2(1-\varphi(e_1+e_2)t)^2}.
    \]
\end{exam}

\begin{lem}\label{lem:weighted_group}
Let $G$ be an Abelian group, $\phi:\Z^d\to G$ be a homomorphism, and $\Delta$ be a half-open simplex with vertices in $\Z^d$. Suppose $w\in R[x_1,\ldots,x_d]$ is a homogeneous linear polynomial. Then 
the $(\phi,w)$-Ehrhart series 
$$\mathcal E(\Delta,\phi,w;t):=\sum_{n\geq 0} \left(\sum_{\alpha\in n\Delta\cap \Z^d}w(\alpha)\phi(\alpha)\right)t^n$$
can be written as a rational function in $R[G](t)$ with denominator $\prod (1-\phi(v_i)t)^2$. 
\end{lem}
\begin{proof}
Since $w$ is linear and homogeneous, $w(x)=a_1x_1+\ldots+a_dx_d$.
As in Example~\ref{ex:triangle-polynomial}, let $q$ be a variable and let $H=G_q\times G$ be the direct product of the group of $R$-powers of $q$ and $G$. Let $\psi:\Z^d\to H$ be defined as $\psi(\alpha)=q^{w(\alpha)}\phi(\alpha)$.
Repeating the argument from Example~\ref{ex:triangle-polynomial}, we can formally differentiate the rational function expression for $\mathcal{E}(\Delta,\psi;t)$ with respect to $q$ and evaluate at $q=1$ to get the claimed result.
\end{proof}

Although the rational function in Example~\ref{ex:triangle-polynomial} had cancellation in the denominator, this does not always happen.

\begin{exam}\label{ex:polynomial-weights}
Let $\Delta=\conv((1,0),(2,0),(2,1))\subset \R^2$, $w=x_1$, $G$ is the group of (integer) powers of $q$ and $\phi(x_1,x_2)=q^{x_1+x_2}$. Then
\[\begin{aligned}
\mathcal E(\Delta,\phi,w;t) &=\sum_{n\geq 0} \left(\sum_{\alpha\in n\Delta\cap \Z^2}w(\alpha)\phi(\alpha)\right)t^n
= \sum_{n\geq 0} \left(\sum_{(x_1,x_2)\in n\Delta\cap \Z^2}x_1q^{x_1+x_2}\right)t^n\\
&= \sum_{n\geq 0} \left(\sum_{x_1=n}^{2n}\sum_{x_2=0}^{x_1-n}x_1q^{x_1+x_2}\right)t^n 
=\frac{5 q^6 t^3-4 q^5 t^2-3 q^4 t^2-3 q^3 t^2+2 q^3 t+2 q^2 t+q t}{(1-q t)^2 (1-q^2 t)^2 (1-q^3 t)^2}.
\end{aligned}\]
This example shows that the denominator of the rational function cannot be simplified.
\end{exam}

\begin{thm}\label{thm:weighted-phi}
Let $G$ be an Abelian group, $\phi:\Z^d\to G$ be a homomorphism, and $\Delta$ be a half-open simplex with vertices in $\Z^d$. Suppose $w\in R[x_1,\ldots,x_d]$ is a polynomial of degree~$m$. Then 
the $(\phi,w)$-Ehrhart series 
$$\mathcal E(\Delta,\phi,w;t):=\sum_{n\geq 0} \left(\sum_{\alpha\in n\Delta\cap \Z^d}w(\alpha)\phi(\alpha)\right)t^n$$
can be written as a rational function in $R[G][[t]]$ with denominator $\prod (1-\phi(v_i)t)^{m+1}$. 
\end{thm}
\begin{proof}
We sketch the proof following the approach of Lemma \ref{lem:weighted_group}. Again, it is enough to prove the theorem when $w$ is a monomial of degree $m$ as we can take linear combinations of the corresponding series and rational functions with coefficients in $R$ for different monomials.

We assume that $w=y_1\ldots y_m$ where each $y_i$ is a variable $x_j$ and we treat each $y_i$ as linear function returning $x_j$. We introduce $m$ different auxiliary variables $q_1,\ldots, q_m$. Let $H=G_{q_1}\times\ldots \times G_{q_m}\times G$ and $\psi(\alpha)=\left(\prod q_i^{y_i(\alpha)}\right)\phi(\alpha)$ where the groups $G_{q_i}$ are the groups of $R$-powers of $q_i$, respectively. We use Corollary~\ref{cor:ehrhart_group} to write down the power series
$$\mathcal E(\Delta,\psi;t)=\sum_{n\geq 0}\left(\sum_{\alpha\in n\Delta\cap \Z^d}\left(\prod q_i^{y_i(\alpha)}\right)\phi(\alpha)\right)t^n$$
as rational function with denominator $\prod (1-\left(\prod q_j^{y_j(v_i)} \right)\phi(v_i)t)$.

Differentiating with respect to each $q_i$ and replacing each $q_i$ with $1\in R$ finishes the proof.
\end{proof}

As with the case without polynomial weights, Theorem \ref{thm:weighted-phi} generalizes to lattice polytopes through half-open triangulations. The weighted Ehrhart theory with polynomial weights (see \cite{Weighted_Bajo_etal}) can be recovered by using the trivial group $G$ with ring $R=\R$. Similarly, we again have a Brion's theorem with polynomial weights. For a rational polyhedron $Q \subset \R^d$, homomorphism $\varphi : \Z^d \rightarrow G$, and polynomial weight $w$ with coefficients in a commutative ring $R$, let the $w$-weighted $\varphi$-integer point transform be
\[
\sigma_R(Q,\varphi,w;\mathbf{z}) = \sigma(Q,\varphi,w;\mathbf{z}) =  
\sum_{\alpha \in Q \cap \Z^d} w(\alpha)\varphi(\alpha)\mathbf{z}^\alpha.
\]
It is a multivariate rational function in variables $z_1,\dots, z_d$ with coefficients in $R[G]$.
Applying the homomorphisms and differentiation used before we can also obtain:

\begin{thm}\label{thm:weighte-phi-Brion}
Let $P$ be a rational polytope.  We have the following identity of rational functions:
\[
\sigma(P,\varphi,w;\mathbf{z}) =  \sum_{v \text{ vertex of }P} \sigma(K_{v,P}, \varphi,w; \mathbf{z}).
\]
\end{thm}

\begin{proof}[Sketch of the proof]
It is enough to prove the theorem when $w$ is a monomial. Similarly to the proof of Theorem~\ref{thm:weighted-phi}, we apply Theorem~\ref{thm:phi-Brion} to the appropriate group $G_{q_1}\times \ldots \times G_{q_t}\times G$ and the corresponding homomorphism $\psi$ defined as $\psi(\alpha)=\left( \prod q_i^{y_i(\alpha)}\right)\phi(\alpha)$. After that we differentiate with respect to each $q_i$ and replace $q_i$ with $1\in R$.
\end{proof}

\begin{exam}
We modify Example~\ref{exam:brion}. Let $P=[0,1]^2$, $R=\Z$, $G=\langle \tau \ | \ \tau^3=1\rangle$, $\phi:\Z^2 \to G$ be given by $\varphi(i,j)=\tau^{i+j}$, and let the weight be $w(\alpha_1,\alpha_2)=\alpha_1+2\alpha_2$.

The lattice points of $P$ are $(0,0),(1,0),(0,1),(1,1)$ so we directly have
\[
\sigma(P,\phi,w; \mathbf{z})=\tau z_1 + 2\tau z_2 + 3\tau^2 z_1z_2.
\]
Now compute $\sigma(K_{v,P},\phi,w;\mathbf{z})$ for the four vertices of $P$. The cone  generated at $(0,0)$ consists of all points in the nonnegative quadrant so we can write
\[
\sigma(K_{(0,0),P},\phi,w;\mathbf{z})=\sum_{\alpha_1,\alpha_2\geq 0}(\alpha_1+2\alpha_2)\tau^{\alpha_1+\alpha_2}z_1^{\alpha_1}z_2^{\alpha_2}.
\]

Using that 
\[\sum_{\alpha\geq 0} \tau^\alpha z_i^\alpha=\frac{1}{1-\tau z_i} \quad \text{and} \quad \sum_{\alpha\geq 0} \alpha \tau^\alpha z_i^\alpha=\frac{\tau z_i}{(1-\tau z_i)^2} \quad \]
we obtain
\[
\sigma(K_{(0,0),P},\phi,w;\mathbf{z})=\frac{\tau z_1}{(1-\tau z_1)^2(1-\tau z_2)}+\frac{2\tau z_2}{(1-\tau z_1)(1-\tau z_2)^2}=\frac{\tau z_1 +2\tau z_2-3\tau^2z_1z_2}{(1-\tau z_1)^2(1-\tau z_2)^2}.
\]

Applying a similar approach we get
\[
\sigma(K_{(1,0),P},\phi,w;\mathbf{z})=\sum_{\alpha_1\leq 1,\alpha_2\geq 0} (\alpha_1+2\alpha_2)\tau^{\alpha_1+\alpha_2}z_1^{\alpha_1}z_2^{\alpha_2}=\frac{-2+\tau z_1+\tau z_1z_2}{(1-\tau^2 z_1^{-1})^2(1-\tau z_2)^2},
\]
\[
\sigma(K_{(0,1),P},\phi,w;\mathbf{z})=\sum_{\alpha_1\geq 0,\alpha_2\leq 1} (\alpha_1+2\alpha_2)\tau^{\alpha_1+\alpha_2}z_1^{\alpha_1}z_2^{\alpha_2}=\frac{-4+3\tau z_1+2\tau z_2-3\tau^2 z_1z_2}{(1-\tau z_1)^2(1-\tau^2 z_2^{-1})^2},
\]
and
\[
\sigma(K_{(1,1),P},\phi,w;\mathbf{z})=\sum_{\alpha_1\leq 1,\alpha_2\leq 1} (\alpha_1+2\alpha_2)\tau^{\alpha_1+\alpha_2}z_1^{\alpha_1}z_2^{\alpha_2}=\frac{6-5\tau z_1-4\tau z_2+3\tau^2 z_1z_2}{(1-\tau^2 z_1^{-1})^2(1-\tau^2 z_2^{-1})^2}.
\]

Adding the four $w$-weighted $\phi$-integer point transforms for the vertex cones give $\sigma(P,\phi,w;\mathbf z)$.
\end{exam}

\section{Equivariant $\varphi$-Ehrhart theory} 
\label{sec:phiEquivariantEhrhart}
 
Suppose a finite group $H$ is acting linearly on the lattice $\Z^d$ and a lattice polytope $P \subset \Z^d$ is invariant under the $H$ action. The action of $H$ induces a natural permutation action on the set of lattice points in $nP \cap \Z^d$ for each integer $n\geq 0$. 
In \cite{StapledonEquivariant} Stapledon introduced the equivariant Ehrhart series as follows. For $n\in \N$, consider the map $H \rightarrow \N$ that sends a group element $h\in H$ to the number of fixed points of $h$ among the lattice points of $nP$.  This is a class function on $H$, i.e.,\ conjugate group elements take the same value.  If we fix $h\in H$, we can think of this as a function on $n$, which is the usual Ehrhart quasipolynomial of the rational subpolytope of $P$ consisting of points fixed by $h$. For example, the identity group element gives us the usual Ehrhart polynomial of the polytope $P$. We can use representation theory to get a unified expression for the Ehrhart polynomial or Ehrhart series that works for all elements $h\in H$.

Indeed, $H$ acts on the finite set $P \cap \Z^d$ by permuting the lattice points, and the \emph{permutation character} $\chi_P \colon H \to \mathbb{C}$ is the number of fixed points
$$\chi_P(h) = \#\{ m \in P \cap \Z^d \mid h \cdot m = m \}.$$
This function $\chi_P$ lies in the ring of class functions of $H$. It captures symmetry properties of the set of lattice points in $P$, for example, the orbit structures via the Burnside lemma
    $$\# \text{orbits of } H \text{ on } P \cap \Z^d = \frac{1}{|H|} \sum_{h \in H} \chi_P(h).$$

Let $\mathop{Rep}_\C(H)$ be the {\em representation ring} of $H$.  Additively it is the free Abelian group generated by isomorphism classes of linear representations of $H$ over $\C$, modulo the relations $[V] + [W] = [V \oplus W]$. Multiplication is given by tensor product of representations: $[V]\cdot[W] = [V \otimes W]$.  There is an algebra isomorphism from $\mathop{Rep}_\C(H)$ to $C_\C(H)$, the $\C$-valued class functions on $H$, which sends a representation to its character.  
Let $\rho : H \rightarrow \mathop{GL}(\C^d)$ denote the complex representation corresponding to the permutation action of $H$ on $\Z^d$ that we started with. Since $H$  linearly acts on $\Z^d$, the corresponding action of $H$ on $\C^d$ can be represented by integer matrices, so the character of $\rho$ has integer values and lies in the subring $C_\Z(H) \subset C_\C(H)$ of integer valued class functions on $H$. 

Crucially, we stress that the permutation action of $H$ on lattice points in $nP$ also has integer valued characters.
Recall that the permutation character $\chi_P$ encodes the permutation representation of $H$ on $P \cap \Z^d$. But now, in equivariant Ehrhart theory, one studies not just $P \cap \Z^d$, but $nP \cap \Z^d$ for all $n \in \mathbb{Z}_{\geq 1}$, and the sequence of permutation characters $\chi_{nP}$ form a family with Ehrhart-style structure. 

\begin{thm}[Theorem 5.7 of \cite{StapledonEquivariant}]
The function $\N \rightarrow C_\C(H)$ given by $n \mapsto \chi_{nP}$ is a quasi-polynomial in n of degree $\dim(P)$ and period dividing the exponent of $H$. 
\end{thm}
For each $h \in H$, the value of the character $\chi_{nP}(h)$ is equal to the number of lattice points in $n P^h$ where \[P^h := \{ x \in P \mid h \cdot x = x \}\]  is the fixed point polytope, so it is a quasi-polynomial in $n$ by  classical Ehrhart theory.  The beauty of the equivariant Ehrhart theory is that it encapsulates these quasi-polynomials for all $h\in H$ into a single quasi-polynomial over $\mathop{Rep}_\C(H)$ or $C_\C(H)$.

It then follows by Burnside Lemma that the number of orbits of $H$ on $nP \cap \Z^d$ is also a quasi-polynomial in $n$. 
We also get an \emph{equivariant Ehrhart Series}, which is a formal power series in $n$, with coefficients in the character ring:
 $$\sum_{n=0}^\infty \chi_{nP} t^n \in C_\Z(H)[[t]].$$
This is an extension of the classical Ehrhart series.

What happens when we combine equivariant theory, using the group $H$, with our group ring  coefficients $G,R$? Of course, there must be some compatibility between $G$ and $H$. We assume that the group homomorphism $\varphi : \Z^d \rightarrow G$ is $H$-invariant, i.e.,\ $\varphi(h\cdot \alpha) = \varphi(\alpha)$ for all $h \in H, \alpha \in \Z^d$.
We can combine Stapledon's equivariant Ehrhart theory~\cite{StapledonEquivariant} with our homomorphism $\varphi$ as follows. 
As we discussed
all our representations here have integer value characters, so we can tensor the $\Z$-valued class functions with $R[G]$ to work in the ring of $R[G]$-valued class functions on~$H$.

Recall that the permutation character $\chi_P(h)$ of $h \in H$ is the number of integer points in the fixed point polytope $P^h$.  In the weighted case, we count the integral fixed points with weights. For a positive integer $n$, let $\ehr(P,\varphi,H; n)$ be the class function on $H$ given by \[\ehr(P,\varphi,H; n) : h \mapsto \sum_{\alpha \in nP^h} \varphi(\alpha) \in R[G]. \]
On the other hand, consider the subring of the representation ring consisting of those with integer characters, and tensor it with $R[G]$.  In this ring, the equivariant $\varphi$-Ehrhart function $\ehr(P,\varphi,H; n)$ corresponds to 
\[\sum_{\substack{\text{orbit }\mathcal{O}\text{ under } \\ H \text{ action on } nP \cap \Z^d}} \varphi(\mathcal O)[\rho_\mathcal{O}]\]
where $\varphi(\mathcal O):=\varphi(x)$ for any $x \in \mathcal{O}$, and $[\rho_\mathcal{O}]$ is the permutation representation of $H$ acting on the orbit $\mathcal O$.  
The equivariant $\varphi$-Ehrhart series is defined following the classical case as
\[
\sum_{n \geq 0} \ehr(P,\varphi,H; n) t^n.
\]
This is a formal power series whose coefficients are $R[G]$ valued class functions of $H$. We compute this in a small example.  
\begin{exam}
    \label{example:equiv}
    Consider the same setting as in Example~\ref{ex:first example}.  Let $P = [0,1]^2$ be a unit square, 
    let $\varphi: \Z^2 \rightarrow C_3 = \langle \tau \rangle$ be given by $\varphi(x_1)=\varphi(x_2)=\tau$, and let $R=\Z$.  Let $H = C_2$ be the group of two elements acting on $\Z^2$ by reflection across the line $x_1=x_2$, and $\rho : H \rightarrow \mathop{GL}_2(\C)$ be the corresponding complex representation.  We will use $1$ and $\varepsilon$ to denote the two isomorphism classes of irreducible representations of $H$ in $\mathop{Rep}(H)$, with $1$ being the trivial representation, and $\varepsilon^2 = 1$.

We obtain the weighted equivariant Ehrhart series by counting the lattice points in the fundamental domain $\conv\{(0,0),(n,0),(n,n)\}$.  The lattice points on the diagonal are fixed points, each contributing a trivial representation, while the other lattice points come in size two orbits, each contributing $1+\varepsilon$.  From the multivariate Ehrhart series of the fundamental triangle
\(
\sum_{n\geq 0}\sum_{\alpha \in nP \cap \Z^2} x^\alpha t^n = \frac{1}{(1-t)(1-x_1 t)(1-x_1 x_2 t)}
\)
we substitute $x_1,x_2$ with $\tau$ and first count all lattice points with weight $1+\varepsilon$, then subtract away the contribution of $\varepsilon$ by the diagonal, to obtain the equivariant weighted Ehrhart series
\begin{equation}
\label{eqn:weightedEquivExample}
\frac{1+\varepsilon}{(1-t)(1-\tau t)(1-\tau^2 t)} - \frac{\varepsilon}{(1-t)(1-\tau^2 t)} = \frac{1+\varepsilon \tau t}{(1-t)(1-\tau t)(1-\tau^2t)}.
\end{equation}

Alternatively, we subdivide $P$ into two $H$-invariant unimodular triangles:
\[
    T_1 = \conv\{(0,0),(0,1),(1,0)\} \qquad \text{and}\qquad 
    T_2  = \conv\{(1,1),(0,1),(1,0)\}.
\]
We adapt \cite[Proposition~6.1]{StapledonEquivariant} to our case, to obtain the Ehrhart series for $T_1$ and $T_2$ respectively as
 \begin{align*}
 \frac{1}{(1-t)\det(I-\rho \tau t)} &= \frac{1}{(1-t)(1-\tau t)(1-\varepsilon \tau t)}, \text{ and } \\
  \frac{1}{(1-\tau^2 t)\det(I-\rho \tau t)} &= \frac{1}{(1-\tau^2t)(1-\tau t)(1-\varepsilon \tau t)}.
 \end{align*}
The factors $(1-t)$ and $(1-\tau^2t)$ correspond  to the fixed points $(0,0) \in T_1$ with weight~$1$ and $(1,1)\in T_2$ with weight $\tau^2$ respectively. The weighted equivariant Ehrhart series of the square is, via inclusion-exclusion,  the sum of these two rational functions minus \(\frac{1}{(1-\tau t)(1-\varepsilon \tau t)}\) for the common edge between the two triangles.
Multiplying both numerator and denominator by $(1+\varepsilon \tau t)$ and simplifying gives us
 \(\frac{1+\varepsilon \tau t}{(1-t)(1-\tau t)(1-\tau^2t)}
\) as in~\eqref{eqn:weightedEquivExample}.
\end{exam}
 In general, a weighted analogue of \cite[Proposition~6.1]{StapledonEquivariant} for $H$-invariant simplices has the product of $\det(I-\rho \,\varphi(\mathcal{O}) t)$ over all orbits $\mathcal{O}$ of vertices for the denominator and the weighted sum of lattice points in the half-open parallelepiped for the numerator.

\section{Concluding remarks}
We conclude the paper with a few remarks: 

\subsection{Volume using lattice points}

The leading coefficient of the classical Ehrhart polynomial of the $d$-dimensional polytope $P$ is well-known to be $\vol (P)$. We can interpret this property as equality $|nP\cap \Z^d|=\vol(P)n^d+O(n^{d-1})$. If the homomorphism $\phi:\Z^d\longrightarrow G$ is surjective, then Theorem~\ref{thm:volume} establishes that the preimage of each element $g\in G$ contains $\frac{n^d}{|G|}\vol(P)+o(n^d)$ points in $nP\cap \Z^d$.

The same result can be obtained from the point of view of the existence of a bounded transport. For each $g\in G$, the complete preimage $\phi^{-1}(g)\subset \Z^d$ is a sublattice of $\Z^d$ of density $\frac{1}{|G|}$. Each two complete preimages differ by translation and it follows that for each complete preimage, $|nP\cap \phi^{-1}(g)|=\frac{n^d}{|G|}\vol(P)+O(n^{d-1})$ where the error comes from the number of points of $\Z^d$ in a $t$-neighborhood of the boundary of $nP$ for a fixed $t$.

We refer to \cite{frettloeh2022,laczkovich1992} for more background and detailed discussion on the connection between the following two properties: (a) existence of a bounded transport between discrete $S\subset \R^d$ and a lattice, and (b) counting points of $S$ in growing regions.

\subsection{Computational complexity}

Our paper showed that changing Ehrhart functions via homomorphism $\phi : G \to \mathbb{R}$ and using group ring $R[G]$ coefficients extends classical results and shows new ones. Some natural questions about computation arise: How can we compute the $\phi$-Ehrhart series efficiently? For the classical choice $G=\{1\}, R=\Z$ it is known that, if the dimension $d$ is fixed, then the polynomial $\ehr(P;n)$ can be found in polynomial time using Barvinok's algorithm \cite{barvinok-1993:exponential-sums}; but what happens with other choices of group and ring? We suspect that similar results occur in the $\phi$-Ehrhart setting. Similarly, can computation of the Ehrhart series be done directly for other group rings, without first computing the classical Ehrhart series and then taking the $\phi$ map?

If dimension $d$ is part of the input, then deciding whether $|P\cap \Z^d|$ is positive for rational $P$ is NP-complete, and finding parity of $|P\cap \Z^d|$ is NP-complete, which follows from the NP-completeness of 3-SAT \cite{karp}. We expect that a similar property will hold for the complexity of counting lattice points modulo a prime, but is it easier than counting over $\mathbb{Z}$? 

Finally, in addition to the mentioned questions and open problems, we are also interested in further applications of the $\phi$-Ehrhart theory we introduced here and its connections with other generalizations. For example, is there a relation to the harmonics and graded Ehrhart theory studied in~\cite{reiner2024harmonicsgradedehrharttheory}?---\textit{Quo vadis, Ehrhart theory?}

\section{Acknowledgements}
Jes\'us A. De Loera was partially supported by NSF-DMS grants 2348578 and 2434665. Alexey Garber was partially supported by the Simons Foundation. Katharina Jochemko was partially supported by  grant nr 2023-04063 from 
the Swedish research council and the Verg Foundation. Mohamed Omar was partially supported by research funds from York University, and NSERC Discovery Grant \#RGPIN-2025-06304. Josephine Yu was partially supported by NSF-DMS grant 2348701. The authors thank the American Institute of Mathematics and its SQuaREs program for support and hospitality. This research was initiated there. 
The authors also thank Sof{\'\i}a Garz\'on Mora for many discussions on 
the topic and Greta Panova for discussion on computational complexity of lattice point counting.

\bibliographystyle{plain}

\bibliography{bibliography}

%% if you use biblatex then this generates the bibliography
%% if you use some other method then remove this and do it your own way
%\printbibliography

\end{document}

%% file: Figures/unit-square-modular-tau.tex
\begin{figure}
    \centering
\begin{tikzpicture}[scale=1]
   \draw [thick, fill=gray!15] (0,0) rectangle (4,4);
   \draw [thick, fill=gray!30] (0,0) rectangle (3,3);
   \draw [thick, fill=gray!45] (0,0) rectangle (2,2);
   \draw [thick, fill=gray!60] (0,0) rectangle (1,1);

   \draw[thin,opacity=0.2] (-0.5,-0.5) grid (4.5,4.5);

   \draw [thick,<->] (-0.5,0) -- (4.5,0);
   \draw [thick,<->] (0,-0.5) -- (0,4.5);
   
   \foreach \i in {(0.15,0.15),(3.15,0.15),(0.15,3.15),(3.15,3.15),(2.15,1.15),(1.15,2.15),(4.15,2.15),(2.15,4.15)} {
        \node at \i {\tiny{$1$}};
    }

    \foreach \i in {(1.15,0.15),(0.15,1.15),(4.15,0.15),(3.15,1.15),(2.15,2.15),(1.15,3.15),(0.15,4.15),(4.15,3.15),(3.15,4.15)} {
        \node at \i {\tiny{$\tau$}};
    }

    \foreach \i in {(2.23,0.2),(1.23,1.2),(0.23,2.2),(4.23,1.2),(3.23,2.2),(2.23,3.2),(1.23,4.2),(4.23,4.2)} {
        \node at \i {\tiny{$\tau^2$}};
}
\end{tikzpicture}

\caption{The first %four 
dilations of the unit square in $\Z^2$. Each %lattice 
point $v\in\Z^2$ is labeled with the $\varphi(v)$ from Example~\ref{ex:first example}.}\label{fig:modularity example}

\end{figure}

%% file: Figures/chart-tau.tex
\begin{table}[h]
\centering
\scalebox{1}{
\begin{tabular}{|Sc|Sc|Sc|Sc|} \hline
\diagbox{$G$}{$R$}   & $\Z$ & $\Z_2$                                         & $\Z_m$                                                                 \\ \hline
            $\{1\}$  & $|nP \cap \Z^d|$& parity of $|nP \cap \Z^d|$                     & $|nP \cap \Z^d| \pmod m$                                   \\ \hline
            % $C_2$ & $\displaystyle{\sum_{i=0}^1 |nP \cap K_i|\tau^i}$  & parities of each $|nP \cap K_i|$   & $\displaystyle{\sum_{i=0}^1 (|nP \cap K_i| \pmod m)\tau^i}$     \\ \hline
             $C_r$ & $\displaystyle{\sum_{i=0}^{r-1}|nP \cap K_i|}\tau^i$ & parities of each $|nP \cap K_i|$  & $\displaystyle{\sum_{i=0}^{r-1} (|nP \cap K_i| \pmod m)\tau^i}$  \\ \hline
\end{tabular}}
\caption{Interpretations of $\ehr(P,\varphi;n)$ for some pairs of $G$ and $R$.
%finite groups $G$ and rings $R$.
Each uses the homomorphism $\varphi: \Z^d \to G$ sending standard basis vectors to the same generator $\tau\in G$. We use $K_i$ to denote the coset $ie_1 + \ker(\varphi)$ of $\Z^d/\ker(\varphi)$.}\label{tab:group-ring-table}
\end{table}

%% file: Figures/brion-ex-tau.tex
\begin{figure}
    \centering
\begin{tikzpicture}[scale=0.65]
   \fill [gray!30] (0,0) rectangle (4,4);
%   \draw [thick, fill=gray!30] (0,0) rectangle (3,3);
%   \draw [thick, fill=gray!45] (0,0) rectangle (2,2);

   \draw[thin,opacity=0.2] (-0.5,-0.5) grid (4.5,4.5);

   \draw [thick,gray] (-0.5,0) -- (4.5,0);
   \draw [thick,gray] (0,-0.5) -- (0,4.5);

   \draw [very thick, ->] (0,0) -- (4.6,0);
   \draw [very thick,->] (0,0) -- (0,4.6);
   \draw [thick,fill=gray!50] (0,0) rectangle (1,1);

% dots
 \foreach \i in {0.5,1.5,2.5,3.5 } {
           \node at (\i,4.5) {\scalebox{0.5}{\color{gray} $\vdots$}};
    }
 \foreach \i in {0.5,1.5,2.5,3.5 } {
           \node at (4.5,\i) {\scalebox{0.5}{\color{gray} $\cdots$}};
    }

\node at (4.5, 4.5) {\scalebox{0.5}{\color{gray} $\Ddots$}};
\node at (-0.3, -0.3) {\phantom{\scalebox{0.5}{\color{gray} $\ddots$}}};

%   \foreach \i in {(-0.16,-0.19),(2.84,-0.19),(-0.16,2.82),(2.84,2.82),(1.84,0.82),(0.84,1.82),(3.84,1.82),(1.84,3.82)} {
   \foreach \i in {(0.16,0.2), (3.16, 0.2), (0.16, 3.2), (3.16, 3.2), (2.16, 1.2), (1.16, 2.2), (4.16, 2.2),(2.16, 4.2)} {
           \node at \i {\tiny{$1$}};
    }

%    \foreach \i in {(0.82,-0.17),(-0.18, 0.83),(3.82, -0.17), (2.82 ,0.83), (1.82, 1.83), (0.82, 2.83), (-0.18, 3.83), (3.82, 2.83), (2.82, 3.83)} {
   \foreach \i in {(0.2,1.19),(3.2, 1.19),(0.2, 4.18),(3.2, 4.18),(2.2, 2.18),(1.2, 3.18),(4.2, 3.18),(1.2, 0.18),(4.2,0.18)} {
        \node at \i {\tiny{$\tau$}};
    }

%    \foreach \i in {(1.79, -0.21),(0.82, 0.82),(-0.18, 1.82),(3.82, 0.82),(2.82, 1.82),(1.82, 2.82),(0.82, 3.82),(3.82, 3.82)} {
   \foreach \i in {(0.32, 2.25),(3.32, 2.25),(1.32, 1.25),(4.32, 1.25),(2.32, 3.25),(1.32, 4.25),(4.32, 4.25),(2.32, 0.25)} {
        \node at \i {\tiny{$\tau^2$}};
     }

\end{tikzpicture}\hfill
\begin{tikzpicture}[scale=0.65]
   \fill [gray!30] (0,0) rectangle (4,4);
%   \draw [thick, fill=gray!30] (0,0) rectangle (3,3);
%   \draw [thick, fill=gray!45] (0,0) rectangle (2,2);
%   \draw [thick, fill=gray!60] (0,0) rectangle (1,1);

   \draw[thin,opacity=0.2] (-0.5,-0.5) grid (4.5,4.5);

   \draw [thick,gray] (-0.5,0) -- (4.5,0);
   \draw [thick,gray] (3,-0.5) -- (3,4.5);

   \draw [very thick,<-] (-0.6,0) -- (4,0);
   \draw [very thick,->] (4,0) -- (4,4.6);
   \draw [thick,fill=gray!50] (3,0) rectangle (4,1);

% dots
 \foreach \i in {0.5,1.5,2.5,3.5 } {
           \node at (\i,4.5) {\scalebox{0.5}{\color{gray} $\vdots$}};
    }
 \foreach \i in {0.5,1.5,2.5,3.5 } {
           \node at (-0.3,\i) {\scalebox{0.5}{\color{gray} $\cdots$}};
    }

\node at (-0.3, 4.5) {\scalebox{0.5}{\color{gray} $\ddots$}};
\node at (-0.3, -0.3) {\phantom{\scalebox{0.5}{\color{gray} $\ddots$}}};

%   \foreach \i in {(-0.16,-0.19),(2.84,-0.19),(-0.16,2.82),(2.84,2.82),(1.84,0.82),(0.84,1.82),(3.84,1.82),(1.84,3.82)} {
   \foreach \i in {(0.16,0.2), (3.16, 0.2), (0.16, 3.2), (3.16, 3.2), (2.16, 1.2), (1.16, 2.2), (4.16, 2.2),(2.16, 4.2)} {
           \node at \i {\tiny{$1$}};
    }

%    \foreach \i in {(0.82,-0.17),(-0.18, 0.83),(3.82, -0.17), (2.82 ,0.83), (1.82, 1.83), (0.82, 2.83), (-0.18, 3.83), (3.82, 2.83), (2.82, 3.83)} {
   \foreach \i in {(0.2,1.19),(3.2, 1.19),(0.2, 4.18),(3.2, 4.18),(2.2, 2.18),(1.2, 3.18),(4.2, 3.18),(1.2, 0.18),(4.2,0.18)} {
        \node at \i {\tiny{$\tau$}};
    }

%    \foreach \i in {(1.79, -0.21),(0.82, 0.82),(-0.18, 1.82),(3.82, 0.82),(2.82, 1.82),(1.82, 2.82),(0.82, 3.82),(3.82, 3.82)} {
   \foreach \i in {(0.32, 2.25),(3.32, 2.25),(1.32, 1.25),(4.32, 1.25),(2.32, 3.25),(1.32, 4.25),(4.32, 4.25),(2.32, 0.25)} {
        \node at \i {\tiny{$\tau^2$}};
}

\end{tikzpicture}\hfill
\begin{tikzpicture}[scale=0.65]
   \fill [gray!30] (0,0) rectangle (4,4);
%   \draw [thick, fill=gray!30] (0,0) rectangle (3,3);
%   \draw [thick, fill=gray!45] (0,0) rectangle (2,2);
%   \draw [thick, fill=gray!60] (0,0) rectangle (1,1);

   \draw[thin,opacity=0.2] (-0.5,-0.5) grid (4.5,4.5);

   \draw [thick,gray] (-0.5,3) -- (4.5,3);
   \draw [thick,gray] (0,-0.5) -- (0,4.5);

   \draw [very thick, ->] (0,4) -- (4.6,4);
   \draw [very thick,->] (0,4) -- (0,-0.6);
   \draw [thick,fill=gray!50] (0,3) rectangle (1,4);

% dots
 \foreach \i in {0.5,1.5,2.5,3.5 } {
           \node at (\i,-0.3) {\scalebox{0.5}{\color{gray} $\vdots$}};
    }
 \foreach \i in {0.5,1.5,2.5,3.5 } {
           \node at (4.5,\i) {\scalebox{0.5}{\color{gray} $\cdots$}};
    }

\node at (4.5, -0.3) {\scalebox{0.5}{\color{gray} $\ddots$}};

%   \foreach \i in {(-0.16,-0.19),(2.84,-0.19),(-0.16,2.82),(2.84,2.82),(1.84,0.82),(0.84,1.82),(3.84,1.82),(1.84,3.82)} {
   \foreach \i in {(0.16,0.2), (3.16, 0.2), (0.16, 3.2), (3.16, 3.2), (2.16, 1.2), (1.16, 2.2), (4.16, 2.2),(2.16, 4.2)} {
           \node at \i {\tiny{$1$}};
    }

%    \foreach \i in {(0.82,-0.17),(-0.18, 0.83),(3.82, -0.17), (2.82 ,0.83), (1.82, 1.83), (0.82, 2.83), (-0.18, 3.83), (3.82, 2.83), (2.82, 3.83)} {
   \foreach \i in {(0.2,1.19),(3.2, 1.19),(0.2, 4.18),(3.2, 4.18),(2.2, 2.18),(1.2, 3.18),(4.2, 3.18),(1.2, 0.18),(4.2,0.18)} {
        \node at \i {\tiny{$\tau$}};
    }

%    \foreach \i in {(1.79, -0.21),(0.82, 0.82),(-0.18, 1.82),(3.82, 0.82),(2.82, 1.82),(1.82, 2.82),(0.82, 3.82),(3.82, 3.82)} {
   \foreach \i in {(0.32, 2.25),(3.32, 2.25),(1.32, 1.25),(4.32, 1.25),(2.32, 3.25),(1.32, 4.25),(4.32, 4.25),(2.32, 0.25)} {
        \node at \i {\tiny{$\tau^2$}};
}

\end{tikzpicture}\hfill
\begin{tikzpicture}[scale=0.65]
   \fill [gray!30] (0,0) rectangle (4,4);
%   \draw [thick, fill=gray!30] (0,0) rectangle (3,3);
%   \draw [thick, fill=gray!45] (0,0) rectangle (2,2);
%   \draw [thick, fill=gray!60] (0,0) rectangle (1,1);

   \draw[thin,opacity=0.2] (-0.5,-0.5) grid (4.5,4.5);

   \draw [thick,gray] (-0.5,3) -- (4.5,3);
   \draw [thick,gray] (3,-0.5) -- (3,4.5);

   \draw [very thick,->] (4,4) -- (-0.6,4);
   \draw [very thick,->] (4,4) -- (4,-0.6);
   \draw [thick,fill=gray!50] (3,3) rectangle (4,4);

% dots
 \foreach \i in {0.5,1.5,2.5,3.5 } {
           \node at (\i,-0.3) {\scalebox{0.5}{\color{gray} $\vdots$}};
    }
 \foreach \i in {0.5,1.5,2.5,3.5 } {
           \node at (-0.3,\i) {\scalebox{0.5}{\color{gray} $\cdots$}};
    }

\node at (-0.3, -0.3) {\scalebox{0.5}{\color{gray} $\Ddots$}};

%   \foreach \i in {(-0.16,-0.19),(2.84,-0.19),(-0.16,2.82),(2.84,2.82),(1.84,0.82),(0.84,1.82),(3.84,1.82),(1.84,3.82)} {
   \foreach \i in {(0.16,0.2), (3.16, 0.2), (0.16, 3.2), (3.16, 3.2), (2.16, 1.2), (1.16, 2.2), (4.16, 2.2),(2.16, 4.2)} {
           \node at \i {\tiny{$1$}};
    }

%    \foreach \i in {(0.82,-0.17),(-0.18, 0.83),(3.82, -0.17), (2.82 ,0.83), (1.82, 1.83), (0.82, 2.83), (-0.18, 3.83), (3.82, 2.83), (2.82, 3.83)} {
   \foreach \i in {(0.2,1.19),(3.2, 1.19),(0.2, 4.18),(3.2, 4.18),(2.2, 2.18),(1.2, 3.18),(4.2, 3.18),(1.2, 0.18),(4.2,0.18)} {
        \node at \i {\tiny{$\tau$}};
    }

%    \foreach \i in {(1.79, -0.21),(0.82, 0.82),(-0.18, 1.82),(3.82, 0.82),(2.82, 1.82),(1.82, 2.82),(0.82, 3.82),(3.82, 3.82)} {
   \foreach \i in {(0.32, 2.25),(3.32, 2.25),(1.32, 1.25),(4.32, 1.25),(2.32, 3.25),(1.32, 4.25),(4.32, 4.25),(2.32, 0.25)} {
        \node at \i {\tiny{$\tau^2$}};
}
\end{tikzpicture}

\caption{The four cones $K_{v,P}$ generated at the vertices $v$ of $P = [0,1]^2$.
	The square $[0,1]^2$ is outlined and each lattice point $(i,j)$ is labeled by $\phi(i,j) = \tau^{i+j}$.}\label{fig:brion example}

\end{figure}